\documentclass{amsart}

\usepackage[english]{babel}

\title[Strong semisimplicity and tangents]
{Bouligand-Severi $k$-tangents 
and strongly semisimple MV-algebras}

\author[L.M. Cabrer]{Leonardo Manuel Cabrer}

\address[l.cabrer@disia.unifi.it]{
{\sc Leonardo Manuel Cabrer}: Department
 of Computer Science, Statistics and Applications ``Giuseppe Parenti'',
University of Florence,
Viale Morgagni 59 --
I-50134, Florence -
Italy}

\thanks{AMS2010. Primary: 06D35. %MV-algebras
Seondary:
54C40, % Algebraic properties of function spaces
49J52, %non smooth analysis
52B05, %Combinatorial properties
03B50%many valued logic
}


\thanks{\noindent This research was supported by a Marie Curie Intra European Fellowship within the 7th European Community Framework Program (ref. 299401-FP7-PEOPLE-2011-IEF)}

\usepackage{amssymb}

\newtheorem{theorem}{Theorem}[section]
\newtheorem{lemma}[theorem]{Lemma}

\theoremstyle{definition}
\newtheorem{definition}[theorem]{Definition}

\newcommand{\restrict}{\,{\mathbin{\vert\mkern-0.3mu\grave{}}}\,}

\DeclareMathOperator{\Rn}{{\mathbb R^{\it n}}}
\DeclareMathOperator{\McNn}{\mathcal M([0,1]^{\it n})}
\DeclareMathOperator{\McN}{\mathcal M}
\DeclareMathOperator{\conv}{\rm conv}
\DeclareMathOperator{\den}{\rm den}
\DeclareMathOperator{\I}{[0,1]}
\DeclareMathOperator{\cube}{[0,1]^{\it n}}
\DeclareMathOperator{\aff}{\rm aff}
\DeclareMathOperator{\relint}{\rm relint}
\DeclareMathOperator{\Zed}{\ensuremath{\mathbb{Z}}}

\begin{document}

  \begin{abstract}
   An algebra $A$  is said to be strongly semisimple if
   every principal congruence of $A$ is
   an intersection of maximal congruences. 
   We give a geometrical characterisation of
   strongly semisimple MV-algebras in terms of
   Bouligand-Severi $k$-tangents. The latter
   are  a $k$-dimensional generalisation       
   of the classical Bouligand-Severi tangents.
\end{abstract}

\maketitle

Each semisimple MV-algebra $A$
 is isomorphic to  a separating MV-algebra of 
 continuous $\I$-valued
 maps defined on a compact Hausdorff space $X$,
 which turns out to be homeomorphic to the maximal 
 spectral space of $A$. 
In the particular case when  $A$ is  $n$-generated,
 it is no loss of generality to assume that
 $X$  is a compact 
  subset of $[0,1]^n$ and $A$
is   isomorphic to the MV-algebra  $\McN(X)$ of McNaughton maps of $[0,1]^n$ restricted to $X$ (see \cite[Thm. 3.6.7]{CiDoMu2000}).

Following Dubuc and Poveda \cite{DuPo2010},
we say that an MV-algebra $A$  is
{\it strongly semisimple}  if all its principal
quotients are semisimple. 
In  \cite{BuMu201X}, Busaniche and Mundici characterise those sets $X\subseteq \I^2$
 having the property that
  the  MV-algebra $\McN(X)$ to be strongly semisimple. 
Their result  (Theorem \ref{Thm:BusMun0} below) 
is formulated in terms of the following classical notion
(see \cite{Bo1930,Se1927,Se1931};
also see  
 \cite[p.16]{BoGrWa2009}, \cite[pp.14 and 133]{Mo2006} 
 for modern
 reformulations):

 \begin{definition}
 \label{definition:severi}
 Let  $\emptyset\not=X\subseteq \Rn$ and
 $x\in \Rn$.
 A {\it Bouligand-Severi tangent  
of $X$ at $x$}  is a unit vector $u \in \Rn$
such that $X$ contains a sequence $x_1,x_2,\ldots$ with the following properties: 
\begin{itemize}
\item[(i)] $x_i\not=x$ for all $i$,
 \item[(ii)] $\lim_{i\to \infty } x_i = x,$ 
\,\,and 
\item[(iii)] $\lim_{i\to \infty } {(x_i-x)}/{||x_i-x||} =u.$
\end{itemize}
   \end{definition}

\begin{theorem}\cite[Thm. 2.4]{BuMu201X} 
\label{Thm:BusMun0}
Let $X\subseteq \I^2$ be a closed set. 
The MV-algebra 
$\McN(X)$ is not strongly semisimple  iff  
there exist a point $x\in X$,  a unit vector
 $u\in\mathbb{R}^2$, and a real number $\lambda>0$ such that
\begin{itemize}
\item[{\rm (i)}] $u$ is a Bouligand-Severi tangent of $X$ at $x$,
\item[{\rm (ii)}] $\conv(x,x+\lambda u)\cap X=\{x\}$, and 
\item[{\rm (iii)}] the coordinates
of $x$ and $x+\lambda u$ are rational.
\end{itemize}
\end{theorem}

In Theorem \ref{theorem:aereo} this result will be
generalised to  all finitely generated MV-algebras,
using the  higher-order
Bouligand-Severi tangents
 defined in \ref{definition:k-tangent}-\ref{definition:outgoing}.

 \section{Preliminaries}
\subsection{Semisimple MV-algebras}
We refer the reader to \cite{CiDoMu2000} for background on MV-algebras. 
We let $\McNn$ denote the MV-algebra of piecewise (affine) linear
continuous functions ${f \colon [0,1]^{n}\to [0,1]}$, such that each
linear
piece of $f$ has integer coefficients, with the pointwise
operations of  the standard  MV-algebra $[0,1]$.
 $\McNn$ is the free $n$-generator MV-algebra.
More generally, for any nonempty subset ${X\subseteq [0,1]^{n}}$ we
denote by $\McN(X) $ the MV-algebra of restrictions to $X$ of
the functions in $\McNn$. For every $f\in \McN(X) $ we let
$Zf=f^{-1}(0)$.

By an {\it ideal}  of
an MV-algebra $A$  we mean the kernel of an (MV-)homo\-morphism.
  An ideal is  {\it principal} if it is singly  generated. 
  For each $a\in A$, the principal ideal $\langle a\rangle$ generated by $a$ is the set $\{b\in A\mid \mbox{for some }m\in\mathbb{Z}_{>0}, b\leq ma\}$. An ideal $I$ is   {\it maximal} if $I\neq A$,
  and whenever $J$ is an ideal such that $I\subseteq J$,
   then $J=I$ or $J=A$.
For each closed set $X\subseteq [0,1]^n$ and   $x\in X$, the set $I_x=\{f\in\McN(X)\mid f(x)=0\}$ is a maximal ideal
 of $\McN(X)$. Moreover, for each maximal ideal $I$ of $\McN(X)$, there exists a uniquely
determined  $x\in X$ such that $I=I_x$

   An MV-algebra $A$  is said to be {\it semisimple} if the intersection of its maximal ideals
 is $\{0\}$.
Each semisimple MV-algebra is isomorphic to  a separating MV-algebra of continuous maps from a compact Hausdorff space into $[0,1]$. In particular, if $A$ is an $n$-generated semisimple MV-algebra then there is a closed set $X\subseteq [0,1]^n$ such that $A\cong\McN(X)$.

An  MV-algebra $A$ is {\it strongly semisimple}
 if every principal ideal   of $A$ is an intersection of 
 maximal ideals of $A$.
Equivalently, $A$ is strongly semisimple if for each $a\in A$, the quotient algebra $A/\langle a\rangle$ is semisimple.

\subsection{Simplicial Geometry}
We refer to \cite{Ew1996}, \cite{RS1972} and \cite{St1967} for
background in elementary polyhedral topology.

For any set $\{v_{0},\dots,v_m\}\subseteq \Rn$, $\conv(v_{0},\dots,v_m)$ denotes its {\it convex hull}. If  $\{v_{0},\dots,v_m\}$ are affinely independent then $S=\conv(v_{0},\dots,v_m)$ is an {\it $m$-simplex}.
For any $V\subseteq \{v_{0},\dots,v_m\}$, the convex hull $\conv(V)$ is called a
{\it face} of $S$. If $|V|=m-1$ then $\conv(V)$ is called a {\it facet} of $S$.

For any $m$-simplex $S=\conv(v_0,\ldots,v_m)\subseteq\Rn$, we let
$\aff S$ denote the {\it affine hull} of $S$,  i.e.,
\begin{align*}
\aff S&=\textstyle\{\sum_{i=0}^{m}\lambda_i v_i\mid \mbox{for some }\lambda_i\in\mathbb{R},\mbox{ with } \sum_{i=0}^{m}\lambda_i=1\}\\
&=v_0+\mathbb{R}(v_1-v_0)+\cdots +\mathbb{R}(v_m-v_{m-1}).
\end{align*}
Further,   we write  $\relint S$ for the
{\it relative interior} of $S$, that is, the topological
 interior of $S$ in the relative topology of  ${\rm aff}(S)$.
For each $v\in\Rn$, $||v||$ denotes the euclidean norm of $v$ in $\Rn$.
For each $0<\delta\in\mathbb R$ and $v\in\Rn$  we use
the notation
$B(\delta,v)=\{w\in\Rn\mid ||v-w||<\delta\}$ for the
open ball of radius $\delta$  centred at $v$. Then
$$
\relint S=\{v\in \Rn\mid \mbox{for some }\delta>0,\  B(\delta,v)\cap\aff S\subseteq S\}.
$$

 For later use
we record here some elementary properties of simplexes.

\begin{lemma}\label{Lem:DesInc}
Let $T\subseteq \Rn$ be a simplex and $F$ a face of $\,T$. If $x\in T$
 then $$x\in \relint F\ \mbox{ iff }\ \  F\mbox{ is the smallest face of }\, T\mbox{ such that }x\in F.$$
 Moreover, for any simplex $S$ contained in $T$ we have
	$$
		S\subseteq F\ \ \mbox{ iff }\ \  F\cap \relint S\neq\emptyset.
	$$
\end{lemma}

\subsubsection*{Notation and Terminology}	
Given  $x\in \Rn$,
a $k$-tuple 
$u = (u_1, \dots , u_k )$  of pairwise
	orthogonal unit vectors in $\Rn$ and a $k$-tuple $\lambda=( \lambda_1,\ldots, \lambda_k)\in\mathbb{R}_{>0}^k$,
 we write 
$$C_{x,u,\lambda}=\conv(x,x+\lambda_1u_1,\ldots, x+\lambda_1u_1+\cdots+\lambda_ku_k),
$$

For any $k$-tuple
 $a=(a_1, \dots , a_k )$  and $l=1,\ldots,k$ we let
 $a(l)$ be an abbreviation of the initial segment  $(a_1,\ldots,a_l)$. 
Then  
  the simplex $C_{x,u(l),\lambda(l)}$ is a face of $C_{x,u,\lambda}$.

\begin{lemma}\cite[Prop. 2.2]{BuMu2007}\label{Lem:CapUsimp}
	For each $x\in \Rn$ and   $k$-tuple $u = (u_1, \dots, u_k )$ of pairwise 
	orthogonal unit vectors in $\Rn$, 
	the family of $(x,u)$-simplexes ordered by inclusion is down-directed. 
	That is, if $C_1$ and $C_1$ are $(x,u)$-simplexes, then $C_1 \cap C_2$ contains  
 an $(x,u)$-simplex.
\end{lemma}

\subsection{Rational Polyhedra and $\Zed$-maps}

An $m$-simplex $S=\conv(v_0,\ldots,v_m)$  is said to be {\it rational} if the coordinates of
each vertex  of $S$ are rational numbers.  
A subset
$P$ of $\Rn$
 is said to be a {\it rational polyhedron} if there are rational
simplexes $T_{1},\ldots,T_{l}$
such that $P=T_{1}\cup\cdots\cup T_{l}$.

Given a rational polyhedron $P$,   a {\it triangulation}  of
$P$ is a simplicial complex $\Delta$ such that $P=\bigcup\Delta$ and each simplex $S\in\Delta$ is rational.
Given triangulations $\Delta$ and $\Sigma$ of $P$, 
we say that $\Delta$ is a {\it subdivision} of $\Sigma$ if every
simplex of $\Delta$ is contained in a simplex of $\Sigma$.

For $v$ a rational point in $\Rn$ we let $\den(v)$ denote the
least common denominator of the coordinates of $v$.
A rational $m$-simplex $S=\conv(v_0,\ldots,v_m)\subseteq R^{n}$ is called {\it regular} if the set of vectors
$\{\den(v_0)(v_0,1),\ldots,\den(v_m)(v_m,1)\}$ is part of a basis of the free abelian
group $\Zed^{n+1}.$
By a {\it regular triangulation} of a rational polyhedron $P$ we understand a
triangulation of $P$ consisting of regular simplexes.

   Given polyhedra  $P\subseteq \Rn$ and $Q\in\mathbb{R}^m$, a map $\eta\colon P\rightarrow Q$ is called a {\it $\Zed$-map}
   if there is a  triangulation $\Delta$  of $P$
   such that on every simplex $T$  of  $\Delta$, $\eta$ coincides
with an affine linear map $\eta_{T}\colon \Rn\rightarrow \mathbb{R}^m$ with integer coefficients.
In particular,
 $f\in\McNn$ iff it is a $\Zed$-map form $[0,1]^n$ to $[0,1]$.

For later use we 
  recall here some properties of regular triangulations and $\Zed$-maps. (See \cite[Chapters 2,3]{Mu2011} for the proofs.)

\begin{lemma}\label{lem:ExtTriang}
	Let $P$ and $Q$ be  rational polyhedra and 
	$\Delta$ a rational triangulation of  $P$. 
	If $Q\subseteq P$,
	there exists a regular triangulation 
	$\Delta'$ of $P$ which  is a subdivision of $\Delta$
	and also satisfies $Q=\bigcup\{S\in\Delta'\mid S\subseteq Q\}$.
\end{lemma}

\begin{lemma}\label{lem:PreImag}
	Let $P$ and $Q$ be  rational polyhedra, 
	and $\eta\colon P\to Q$ 
 	a $\mathbb{Z}$-map. If $R$ is
	 a rational polyhedron contained in $Q$,
 	then $\eta^{-1}(R)$ is a rational polyhedron.
\end{lemma}

\begin{lemma}\label{lem:TriangReg} 
	Let $P$ and $Q$ be  rational polyhedra, 
	$\eta\colon P\to Q$ a $\mathbb{Z}$-map and 
	$\Delta$ a triangulation of $P$. 
 	Then there is a regular triangulation $\nabla$ of $P$
  	which is a subdivision of $\Delta$ 
	and has the property that the restriction
   	$\eta\restrict_S$ of $\eta$ to $S$
    	is (affine) linear for each $S\in\nabla$.
\end{lemma}

\begin{lemma}\label{Lem:ZedMapExten} 
Let $\Delta$ be  a regular
triangulation of a polyhedron $P\subseteq \mathbb{R}^{m}$ and $V$ the set of vertices of the simplexes of $\Delta$. 
Suppose the map
   $f\colon V\rightarrow \Rn$  has the property that
 $f(v)$ is a rational vector of $\Rn$ and $\den(f(v))$ divides $\den(v)$ for each $v\in V$.
 Then there exists a unique $\Zed$-map $\mu\colon P\rightarrow
\Rn$ satisfying:
   \begin{itemize}
     \item[{\rm (i)}] $\mu$ is linear on each simplex of $\Delta$,
     \item[{\rm (ii)}] $\mu\restrict{V}=f$.
   \end{itemize}
\end{lemma}

 \section{Strong semisimplicity and Bouligand-Severi tangents} 
 Here we introduce $k$-dimensional Bouligand-Severi tangents,  replacing
the unit vector $u$ of  Definition \ref{definition:severi}
by a $k$-tuple $u=(u_1,\ldots,u_k)$ of pairwise orthogonal unit	vectors in $\mathbb R^n$.
For each   $l\leq k$, let 
\[
\mathsf{p}_{l}\colon \Rn\to \mathbb{R}u_1+\cdots+\mathbb{R}u_{l}
\] denote the orthogonal projection map onto the linear
subspace of $\Rn$ generated by $u_1,\ldots,u_l$.

\begin{definition}
\label{definition:k-tangent}
A $k$-tuple
$u=(u_1,\ldots,u_k)$ of 
pairwise orthogonal unit vectors in $\mathbb R^n$
is said to be a  {\it Bouligand-Severi tangent of $X$ at $x$ of degree $k$}
	(for short,  {\it $u$ is a $k$-tangent of $X$ at $x$}) if
 	$X$ contains a sequence of points  $x_1,x_2,\ldots$
	converging to $x$,  such that
	no  vector  $x_i-x$ lies in  
	$\mathbb{R}u_1+\cdots+\mathbb{R}u_k$ 
 and  upon defining  ${x_i^1= {(x_i-x)}/{||x_i-x||}}$ and
 inductively,  
			$$
				x_i^{l}=\frac{x_i-x -\mathsf{p}_{l-1}(x_i-x)}{||x_i-x -\mathsf{p}_{l-1}(x_i-x)||}\,\,\,\,\,  \mbox{ $(l\leq k)$},
			$$
it follows that   
			$\lim_{i\to \infty }x^s_i =  u_s,\,\,\,$
			for each $s\in\{1,\ldots,k\}$.
The   sequence $x_1,x_2,\ldots$ is said to   
{\it determine}  $u$. 
\end{definition}

Conditions
  (ii) and (iii) in  Theorem~\ref{Thm:BusMun}
have the following generalisation:

\begin{definition}
\label{definition:outgoing}
A $k$-tangent  $u=(u_1,\ldots,u_k)$  of 
 $X\subseteq \mathbb R^n$    at   $x$ 
 {\it  is rationally outgoing}
 if  there is a rational simplex $S$,
  together with a
face $F\subseteq S$  and
  a $k$-tuple
  $\lambda=(\lambda_1,\ldots,\lambda_k)\in\mathbb{R}_{>0}$  
  such that
$
S\supseteq C_{x,u,\lambda},\,\,\,F\not\supseteq C_{x,u,\lambda}\,\,\,\mbox{and}\,\,\,F\cap X = S\cap X.
$
\end{definition}

\subsubsection*{Remarks}
When $k=1$, Definition \ref{definition:k-tangent}
 amounts to the classical 
 Definition \ref{definition:severi} of a 
Bouligand-Severi  tangent of
  a closed set in $\mathbb{R}^n$. 
Any subsequence of $x_0,x_1,\dots$ also 
determines the tangent $u$. 
Further,
 if $u=(u_1,\ldots,u_k)$ is a $k$-tangent of $X$ at $x$ then ${u(l)=(u_1,\ldots,u_l)}$ is an $l$-tangent of $X$ at $x$ for each $l\in\{1,\ldots,k\}$.

If $u$ is a rationally outgoing  $k$-tangent of $X\subseteq \Rn$ then,
trivially,  $k<n$. In particular if $n=2$ then necessarily $k$ is equal to 1, 
and there is  $\epsilon>0$ such that 
$\conv(x,x+\epsilon u_1)$ is a rational polyhedron and $X\cap\conv(x,x+\epsilon u_1)=\{x\}$.
The main result of \cite{BuMu201X} (Theorem~\ref{Thm:BusMun0}
above) can now be restated as follows:

\begin{theorem}\label{Thm:BusMun}
Let $X\subseteq \I^2$ be a closed set. The MV-algebra 
$\McN(X)$ is strongly semisimple  iff  $X$ 
does not have a rationally outgoing $1$-tangent. 
\end{theorem}

The main result of our paper is the following generalisation of Theorem~\ref{Thm:BusMun}:

\begin{theorem}
\label{theorem:aereo}
 For any  closed $X\subseteq \cube$
 the following conditions are equivalent:
 \begin{itemize}
 \item[{\rm (i)}] The MV-algebra $\McN(X)$ is strongly semisimple.
 \item[{\rm (ii)}] For no  $k=1,\ldots,n-1,\,\,\,$
 $X$ has a rationally outgoing $k$-tangent.
 \end{itemize}
\end{theorem}

Each direction of the equivalence 
in Theorem~\ref{theorem:aereo} depends
 on a key property of rationally outgoing  $k$-tangents.
Accordingly,   the proof
is divided  in two parts,  each of them proved in a separate section.

 \section{Proof of Theorem \ref{theorem:aereo}: 
 ({i})$\,\Rightarrow$({ii})}   
 
\begin{lemma}\label{lem:ZkTangIntersec}
Let $P\subseteq \Rn$ be a polyhedron,  
 $X\subseteq P$ a closed set,
and  ${u=(u_1,\ldots,u_k)}$  a $k$-tangent of $X$    
at   $x.$
 Then $P$ contains an $(x,u)$-simplex.
\end{lemma}

\begin{proof}
Let $y_1,y_2,\ldots$
be a  sequence of elements in $X$ determining the tangent $u$.
Let $\Delta$ be a triangulation of $P$,
 and $S$ a simplex of $P$ such that $\{i\mid y_i\in S\}$ is infinite. 
Since $S$ is closed, $x\in S$.
Let $x_1,x_2,\ldots $ be  the subsequence of $y_1,y_2,\ldots$ whose elements lie in $S$. Then $x_1,x_2,\ldots $
determines the $k$-tangent 
$u=(u_1,\ldots,u_k)$ of 
$X\cap S$ at $x$.

 We will first prove 
\begin{equation}\label{Eq:Aff}
x+\mathbb{R} u_1+\cdots+\mathbb{R} u_k\subseteq\aff S.
\end{equation}
For all  $i$  we have
$$x+\mathbb{R} x_i^1=x+\mathbb{R}\frac{x-x_i}{||x-x_i||}=x+\mathbb{R}(x-x_i)\subseteq\aff S.$$
Since  $\aff S$ is closed then
$x+\mathbb{R} u_1\subseteq\aff S$. 
Suppose  we have proved $$x+\mathbb{R} u_1+\cdots+\mathbb{R} u_{l-1}\subseteq\aff S.$$ 
Let $y\in x+\mathbb{R} u_1+\cdots+\mathbb{R} u_{l-1}$. Then 
\begin{align*}
y+\mathbb{R} x_i^{l}&=y+\mathbb{R}
\frac{x-x_i-\mathsf{p}_{l-1}(x-x_i)}{||x-x_i-\mathsf{p}_{l-1}(x-x_i)||}\\[0.05cm]
&=y+\mathbb{R}(x-x_i-\mathsf{p}_{l-1}(x-x_i))\\[0.12cm]
&\subseteq y+\mathbb{R}(x-x_i)+\mathbb{R} u_1+\cdots+\mathbb{R} u_{l-1}\subseteq\aff S.
\end{align*}
Again,   $\,\,\,y+\mathbb{R} u_l\subseteq \aff S\,\,\,$ and 
$\,\,\,
x+\mathbb{R} u_1+\cdots+\mathbb{R} u_{l}\subseteq\aff S.
$
This concludes the proof of (\ref{Eq:Aff}).
\smallskip

We shall now find $\lambda=(\lambda_1,\ldots,\lambda_k)\in\mathbb{R}^k_{>0}$ such that $C_{x,u,\lambda}\subseteq S$.
To this purpose
 we will prove the following stronger statement:
\smallskip

\noindent{\it Claim:} 
For each $l\leq k$, there exists $\lambda(l)=(\lambda_1,\ldots,\lambda_l)\in\mathbb{R}^l_{>0}$ such that 
\begin{itemize}
\item[(i)]$C_{x,u(l),\lambda(l)}\subseteq S$, and 
\item[(ii)] if $F$ is a face of $S$ such that $z_l=x+\lambda_1u_1+\cdots+\lambda_lu_l\in F$, then ${C_{x,u(l),\lambda(l)}\subseteq F.}$
\end{itemize}

The proof  is by induction on $l=1,\ldots,k-1$.
\smallskip

{\it Basis Step $(l=1)$:}   

In case  $\,\,x\in\relint S$
 there exists $\epsilon>0$ such that ${B(\epsilon,x)\cap \aff S\subseteq \relint S}$. Then setting $\lambda_1=\epsilon/2$, by { (\ref{Eq:Aff})} we obtain
 $C_{x,u(1),\lambda_1}\subseteq B(\epsilon,x)\cap \aff S \subseteq\relint S,$
from which both (i) and (ii) immediately follow.

In case   $\,\,x$ does not lie in $\relint S$,
 let $H$ be an arbitrary facet of $S$ 
containing  $x$ as an element. 
 Let  $\aff H^+$ be the half-space of $\aff S$ with boundary 
$\aff H$ and containing $S.$  
For each $\rho>0$ we have the inclusion
$x+\rho(x_i-x)\in \aff H^+$. Since   $\aff H^+$ is closed, then $x+\mathbb{R}_{\geq} u_1\subseteq  \aff H^+$.
As a consequence,
\begin{equation}\label{Eq:Hyper}
x+\mathbb{R}_{\geq0} u_1\subseteq\bigcap\{\,\aff G^+\mid  G \mbox{ a facet of $S$
 containing $x$}\}.
\end{equation}

For some  $\epsilon_1>0$ 
the points  $x+\epsilon_1 u_1$ must lie in $ S$. 
(For otherwise some facet  $K$ of $S$  has the property
that  for each $\epsilon>0$, $x+\epsilon u_1\in  \aff S\setminus\aff K^+$, where $\aff K^+$ is the half-space of $\aff S$ with boundary 
$\aff K$ and containing $S$. 
From $x\in S\subseteq \aff H^+$ we get  $x\in S\cap \aff K=K$ contradicting { \eqref{Eq:Hyper}.})
 
Now let $\lambda_1=\epsilon_1/2$. Then, $C_{x,\lambda_1,u_1}\subseteq \conv(x,x+\epsilon_1u_1)\subseteq S$, 
and (i) is settled. 
Let $F$ be a face of $S$ 
containing  $x+\lambda_1u_1$. 
Since $x+\lambda_1u_1$ lies in the relative interior of $\conv(x,x+\epsilon_1u_1)\subseteq S$, { by Lemma~\ref{Lem:DesInc}}\ \ ${C_{x,u_1,\lambda_1}\subseteq\conv(x,x+\epsilon_1u_1)\subseteq F}$. This proves (ii), and concludes the proof of the basis step.
\smallskip

{\it Inductive Step:} 

Assume that our claim holds for $l<k$. 
Then there exists 
$\lambda(l)=(\lambda_1,\ldots,\lambda_l)$ in $\mathbb{R}^l_{>0}$ such that 
$C_{x,u(l),\lambda(l)}\subseteq S$ and if $F$ is a face of $S$ containing  $x+\lambda_1u_1+\cdots+\lambda_lu_l$, then $C_{x,u(l),\lambda(l)}\subseteq F$.

For the rest of the proof let $z_l=x+\lambda_1u_1+\cdots+\lambda_lu_l$ and $F_l$ be the face of $S$ such that $z_l\in\relint F_l$. 

As in the basis step, in case
$F_l=S$, there exists $\epsilon\in \mathbb{R}_{>0}$ such that $$B(\epsilon,z_l)\cap \aff S\subseteq \relint S.$$ Setting $\lambda_{l+1}=\epsilon/2$ and $\lambda(l+1)=(\lambda(l),\lambda_{l+1})$, { by (\ref{Eq:Aff}) }
we get
$$z_{l}+\lambda_{l+1}u_{l+1}\in \aff S\cap B(\epsilon,z_l)\subseteq \relint S.$$ Since $S$ is a simplex and 
$C_{x,u(l),\lambda(l)}\subseteq S$,  then
 $$C_{x,u(l+1),\lambda(l+1)}=\conv(C_{x,u(l),\lambda(l)}\cup \{z_{l}+\lambda_{l+1}u_{l+1}\})\subseteq S,$$ which proves (i) and (ii).

In case  $F_l$ is a proper face of $S$,
 let $H$ be an arbitrary facet of $S$ 
containing  $F_l$.
 Let  $\aff H^+$ be the closed half-space of $\aff S$ with boundary 
$\aff H$  containing $S.$  
From  { \eqref{Eq:Aff},} we obtain
\begin{equation}\label{Eq:Hyper2}
z_l+\mathbb{R}_{>0} u_{l+1}\subseteq\bigcap\{\,\aff G^+\mid  G \mbox{ a facet of $S$
 containing $F_l$}\}.
\end{equation}
Therefore, 
  $z_l+\mathbb{R}_{\geq0}u_{l+1}\cap S\neq\{ z_l\} $.
For otherwise, arguing by way of contradiction,
 there is a facet  $K$ of $S$ such that  ${z_l+\epsilon u_{l+1}\in  \aff S\setminus\aff K^+}$  for each $\epsilon>0$.  Since $z_l\in \aff K^+$ then $z_l\in K$. 
Since $F_l$ is the smallest face of $S$ containing $z_l$,  it follows that  $F_l\subseteq K$, contradicting {  (\ref{Eq:Hyper2}).} 

We have just proved that there exists 
$\epsilon_{l+1}>0$ such that $z_l+\epsilon_{l+1}u_{l+1}\in S$. Letting
 $\lambda_{l+1}=\epsilon/2$ we have
  $$C_{x,u(l+1),\lambda(l+1)}=\conv(C_{x,u(l),\lambda(l)}\cup\{z_{l}+\lambda_{l+1}u_{l+1}\})\subseteq S,$$
which settles (i).
 
 For any face $F$  of $S$ such that $z_l+\lambda_{l+1} u_1\in F$,  { by Lemma~\ref{Lem:DesInc}} we have  $z_l\in\conv(z_l, z_l+\epsilon_{l+1}u_{l+1})\subseteq F$. Therefore,
  $F_l\subseteq F$, and by inductive hypothesis $C_{x,u(l),\lambda(l)}\subseteq F$, whence
   $$C_{x,u(l+1),\lambda(l+1)}=\conv(C_{x,u(l),\lambda(l)}\cup\{z_{l}+\lambda_{l+1}u_{l+1}\})\subseteq F.$$
Having thus proved (ii), the claim is settled and the lemma is proved.
\end{proof}

\bigskip

 \noindent{\it Proof of Theorem~\ref{theorem:aereo}}: ({i})$\Rightarrow$({ii}).
Let $u=(u_1,\ldots,u_k)$ be a rationally outgoing $k$-tangent of $X$ at $x$, with
the intent of proving that $\McN(X)$ is not strongly semisimple. 
{ With reference to Definition
\ref{definition:k-tangent},} let
 $S$ be a rational $k$-simplex together with
 a proper face $F\subseteq S$ and reals
 $\lambda=(\lambda_1,\ldots,\lambda_k)\in\mathbb{R}^k_{>0}$
 such that $C_{x,u,\lambda}\subseteq S$,
 $S\cap X =  F\cap X$,  and 
$C_{x,u,\lambda}\not\subseteq F.$

{ By Lemma~\ref{lem:ExtTriang}} there exists  a regular triangulation  
$\Delta$ of $\cube$ such that $S=\bigcup\{T\in\Delta\mid T\subseteq S\}$ 
and $F=\bigcup\{R\in\Delta\mid R\subseteq F\}$. 
Let $f,g\in\McNn$ be the uniquely determined maps which are
 (affine)  linear over each simplex of 
$\Delta$, and for each vertex $v$ (of a simplex) in
$\Delta$ satisfy the conditions
$$
f(v)=\begin{cases}
0&\mbox{if }v\in F;\\
1&\mbox{otherwise};
\end{cases}
\quad\quad\quad
g(v)=\begin{cases}
0&\mbox{if }v\in S;\\
1&\mbox{otherwise}.
\end{cases}
$$
The existence of $f$ and $g$ follows from { Lemma~\ref{Lem:ZedMapExten}.} Observe that
  $Zg=S$ and $Zf=F$. Then:
\begin{equation}
\label{equation:crux}
Zf\restrict_{X}=X\cap Zf=X\cap F=X\cap  S=X\cap Zg=Zg\restrict_X.
\end{equation}

This proves that $f\restrict_{X}$ belongs to a maximal ideal of $\McN(X)$ iff  $g\restrict_{X}$ does. To complete the proof of ({i})$\Rightarrow$({ii}) it suffices to
settle the following

\medskip
\noindent
{\it Claim.}
 $f\restrict_{X}$ does not belong to the ideal $\langle g\restrict_X\rangle$
generated by $g\restrict_X$.

As a matter of fact, arguing by way of contradiction and letting
the integer
 $m>0$ satisfy $f\restrict _X\leq mg\restrict _X$,
 it follows that  $X$ is contained in 
 the rational polyhedron $P=\{y\in\cube\mid f(y) \leq mg(y)\}$. 
An application of
{ Lemma~\ref{lem:ZkTangIntersec}}
yields  $\lambda'=(\lambda'_1,\ldots,\lambda'_k)\in\mathbb{R}^k_{>0}$ such that  the simplex
$C_{x,u,\lambda'}$
is contained in $P$. { By Lemma~\ref{Lem:CapUsimp}} there exists
 $\epsilon=(\epsilon_1,\ldots,\epsilon_k)\in\mathbb{R}^k_{<0}$
such that $C_{x,u,\epsilon}\subseteq C_{x,u,\lambda}\cap C_{x,u,\lambda'}$.

Since $g$, as well as $mg,$  vanish over $S$, then $f$
vanishes over $C_{x,u,\epsilon}$. 
Therefore, 
${C_{x,u,\epsilon}\subseteq Zf=F}$. 
From this we obtain ${\emptyset\neq\relint C_{x,u,\epsilon}\cap F\subseteq \relint C_{x,u,\lambda} \cap F}$.
 Since $F$ is a face of $S$,  { by Lemma~\ref{Lem:DesInc}} 
$C_{x,u,\lambda}\subseteq F$, which contradicts  our assumption
$C_{x,u,\lambda}\not\subseteq F$.

This completes the proof of the claim, as well
as of the
({i})$\Rightarrow$({ii}) direction of Theorem~\ref{theorem:aereo}.  
  \qed
 
 \section{Proof of Theorem \ref{theorem:aereo}
  ({ii})$\Rightarrow$({i})}

\begin{lemma}\label{lem:ReverseTang}
Let $P\subseteq \mathbb R^n$ be a rational polyhedron and
$X\subseteq P$ a closed set. 
If $\eta\colon P\to \mathbb{R}^2$ is a $\mathbb{Z}$-map such that 
 $\eta(X)$ has a rationally outgoing $1$-tangent, then for some $k\in \{1,\ldots,n-1\},$\,\,\,
$X$ has a rationally outgoing $k$-tangent.
\end{lemma}

\begin{proof}
Let $u\in\mathbb{R}^2$ be a rationally outgoing $1$-tangent of $\eta(X)$ at $x$.
Since $u$ is outgoing there exists $\epsilon>0$ 
  such that both vertices of the segment
  $\conv(x,x+\epsilon u)$ are rational,
   and
   \begin{equation} 
   \label{equation:dopotutto}
   \conv(x,x+\epsilon u)\cap \eta(X)=\{x\}.
   \end{equation} 
   By { Lemma~\ref{lem:PreImag}},  both $\eta^{-1}(\{x\})$ and $\eta^{-1}(\conv(x,x+\epsilon u))$ 
   are rational polyhedra contained in $P$. 
{ By Lemmas~\ref{lem:ExtTriang} and~\ref{lem:TriangReg},}
 there exists a regular triangulation $\Delta$ of $P$ such that $\eta$ is (affinely) linear on each simplex of $\Delta$ and
\begin{eqnarray}
\label{Eq:Deltax}
\eta^{-1}(\{x\})
&=&\bigcup\{R \in\Delta\mid R \subseteq\eta^{-1}(\{x\}) \},\\
\label{Eq:Deltaconv}\eta^{-1}(\conv(x,x+\epsilon u))
&=&\bigcup\{U\in\Delta\mid U\subseteq\eta^{-1}(\conv(x,x+\epsilon u)) \}.
\end{eqnarray}

The rest of the proof is framed in three steps. 
\smallskip

\noindent{\it Step 1:}
Let $x_1,x_2,\ldots$ be
 a sequence of elements of $\eta(X)$ determining the rationally outgoing $1$-tangent
 $u$ of $\eta(X)$ at $x$. 
There exists $T\in\Delta$ such that the set 
${\{i\mid x_i\in \eta( T\cap X)\}}$ is infinite.
The compactness of  $T$ yields
 a sequence $z_1,z_2,\ldots$ in $T$  such that $\eta(z_1),\eta(z_2),\ldots$ is a subsequence of $x_1,x_2,\ldots$ and  
  ${z=\lim_{i\to \infty }z_i}$  exists.
  Since $\eta$ is continuous and $X\cap T$ is closed, 
  we have $\eta(z)=x$ and $z\in X\cap T$.
Since $\eta$ is (affine) linear on $T$, there is 
 a $2\times n$  integer matrix  $A$ and a vector
 $b\in\Zed^2$ such that 
$$\eta(y)=Ay+b\mbox{  for each  }y\in T.$$

  \noindent{\it Step 2:} We claim that there exists $k\in \{1,\ldots,n\}$ together with  $k$
orthogonal unital vectors  $w_1,\ldots,w_k\in\mathbb R^n$ such that
\begin{itemize}
\item[(i)]$Aw_k\neq 0$, and 
\item[(ii)] $Aw_j=0$  for each $j<k$,  
\item[(iii)] $w=(w_1,\ldots,w_k)$ 
 is a $k$-tangent of $X$ at $z$ determined by a subsequence of $z_0,z_1,\ldots$.
\end{itemize}

  \medskip
The vectors  $w_1,\ldots,w_k$ are constructed by 
the following inductive procedure:   
\smallskip

{\it Basis Step:} 

From $\eta(z_i)-\eta(z)\notin \mathbb{R}u$ it follows that 
  $z_i\neq z$ for all $i$, whence  
each vector $z_i^{1}=(z_i-z)/||z_i-z||$  is  well defined. 
Without loss of generality we can 
assume that $z_0^{1},z_1^{1},\ldots$ tends to some unit vector $w_{1}$.
 (If not, using the compactness of  $(n-1)$-dimensional sphere of radius $1$ 
we can take a converging subsequence of $z_0^{1},z_1^{1},\ldots$,
and call $w_{1}$ its limit.) 
Observe that $w_1$ is a $1$-tangent 
of $X$ at $z$. 
If $Aw_1\neq 0$ then $w=w_1$ proves the claim.
Otherwise we proceed inductively.
\smallskip

{\it Inductive Step:} 

Suppose we have obtained  for some $l$, an $l$-tangent 
$w(l)=(w_1,\ldots,w_l)$
 of $X$ at $z$, 
and $Aw_i=0$ for each $i\in\{1,\ldots,l\}$.
Observe that $l<n$. (For otherwise, since  for each 
 $1\leq i< j\leq l$, the vectors $w_i$  and $w_j$ are pairwise orthogonal, then $A$ is the zero matrix, which is
 contradicts $Az+b=\eta(z)\neq \eta(z_i)=Az_i+b$.)
 Since $Aw_i=0$ for each $i\in\{1,\ldots,l\}$, then
$$A(z+\delta_1 w_1+\cdots+\delta_l w_l)+b=A(z)+b=\eta(z)\neq \eta(z_i)  \mbox{ for each $\delta_1,\ldots,\delta_l\in\mathbb{R}$.}
$$
It follows that   $z_i-z\notin \mathbb{R}w_1+\cdots +\mathbb{R}w_l$,
and  the vectors 
 $$z_i^{l+1}=\frac{z_i-z-\mathsf{p}_{l}(z_i-z)}{||z_i-z-\mathsf{p}_{l}(z_i-z)||}$$ are well defined.  
 Taking, if necessary, a subsequence of the $z_i$
 and denoting by  $z_j$   its $j$th element, 
 we can further assume that $\lim_{j\to\infty}z_j^{l+1}=w_{l+1}$  
for some vector $w_{l+1}$.

By construction, the unit vector $w_{l+1}$ is orthogonal to each $w_j$ 
with $j\leq l$,  
 and $w(l+1)=(w_1,\ldots,w_l,w_{l+1})$
  is an $(l+1)$-tangent of $X$ at $z$. 
If $Aw_{l+1}\neq0$ we fix $k=l+1$ and $w=w(l+1)$
 is a $k$-tangent satisfying the properties of the claim. 
If not,
 we proceed inductively. This proves the claim and completes Step~2.
\smallskip

\bigskip

\noindent{\it Step 3:}
Let $w=(w_1,\ldots,w_k)$ be the $k$-tangent of $X$ at $z$ obtained { in Step 2}. 
We will prove that $w$ is rationally outgoing.

Since
 $w$ is also a $k$-tangent of $X\cap T$ at $z$, by 
{ Lemma~\ref{lem:ZkTangIntersec},}
 there exists a $k$-tuple $\gamma=(\gamma_1,\ldots,\gamma_k)\in\mathbb{R}^k_{>0}$ 
 such that $C_{z,w,\gamma}\subseteq T$. 
Since   $Aw_j=0$ for each $j<k$ then
 $\eta(y)=\eta(z)=x$ for each $y\in C_{z,w(k-1),\gamma(k-1)}$. 
We can write
\begin{align*}
0\neq Aw_k&
=\lim_{i\to\infty}Az_i^k
=\lim_{i\to\infty}A\left(\frac{z_i-z-\mathsf{p}_{k-1}(z_i-z)}{||z_i-z-\mathsf{p}_{k-1}(z_i-z)||}\right)\\
	&=\lim_{i\to\infty}\frac{A(z_i)-A(z)}{||z_i-z-\mathsf{p}_{k-1}(z_i-z)||}\\
&=\lim_{i\to\infty}\frac{\eta(z_i)-\eta(z)}{||z_i-z-\mathsf{p}_{k-1}(z_i-z)||}
\cdot \frac{||\eta(z_i)-\eta(z)||}{||\eta(z_i)-\eta(z)||}\\
&=\lim_{i\to\infty}\frac{\eta(z_i)-\eta(z)}{||\eta(z_i)-\eta(z)||}\cdot
\frac{||\eta(z_i)-\eta(z)||}{||z_i-z-\mathsf{p}_{k-1}(z_i-z)||}\,\,.
\end{align*}
Since 
  $0\neq u=\lim_{i\to\infty}(\eta(z_i)-\eta(z))/||\eta(z_i)-\eta(z)||$,
   then for some $c>0$ we can write 
  $$c=\lim_{i\to \infty}
 \frac{||\eta(z_i)-\eta(z)||}{||z_i-z-\mathsf{p}_{k-1}(z_i-z)||}
 \,\,\,\mbox{ and  }\,\,Aw_k=cu.$$

\noindent
Now, let us
    set $\lambda_j=\gamma_j$ for $j< k$,
 and $\lambda_k=\min\{\gamma_k,\epsilon/c\}$. Then ${C_{z,w,\lambda}\subseteq C_{z,w,\gamma} \subseteq T}$.
 For any $y\in C_{z,w,\lambda}$
there is  $0\leq\delta\leq\lambda_k\leq \epsilon/c$ with
\begin{equation}\label{Eq:CalcEta}
\eta(y)=A(z)+\delta A(w_k)+b=\eta(z)+\delta c u=x+\delta c u.
\end{equation}
It follows that  
\begin{equation}\label{Eq:EtaS}
\eta(C_{z,w,\lambda})\subseteq\conv(x,x+\epsilon u).
\end{equation}

To conclude the proof, let
 $S$ be the smallest face of $T$ such that 
 $C_{z,w,\lambda}$ is contained in  $S$.
{ By (\ref{Eq:Deltaconv}) and (\ref{Eq:EtaS})}, $\,\,\,{S\subseteq \eta^{-1}(\conv(x,x+\epsilon u))}$. 
By  { (\ref{Eq:Deltax})}, 
$\,\,S\cap\eta^{-1}(\{x\})$ is a union of faces of $S$. 
The linearity of  $\eta$ on $S$
ensures that  $S\cap\eta^{-1}(\{x\})$ is convex, whence a face of $S$.

Letting $F=S\cap\eta^{-1}(\{x\})$,
{  from (\ref{equation:dopotutto})}, it follows that $S\cap X=F\cap X$.  
{ Moreover by (\ref{Eq:CalcEta})}, $\eta(z+\lambda_1w_1+\cdots+\lambda_kw_k)=\eta(z)+\lambda_k c u\neq x$. Then $C_{z,w,\lambda}\not\subseteq F$.
We have shown 
 that the $k$-tangent
$w=(w_1,\ldots, w_k)$ is rationally outgoing.
This  concludes Step 3 and completes the proof
of the lemma.
 \end{proof}

\noindent{\it Proof of Theorem~\ref{theorem:aereo}}: ({ii})$\Rightarrow$({i}).
By way of contradiction, let $f,g\in\McN(\cube)$ be such that 
$f\restrict_{X}$ 
does not belong to
 the ideal generated by $g\restrict_X$ and that $f\restrict_{X}$ belongs to a maximal ideal of $\McN(X)$ iff $g\restrict_{X}$  does. Let $A$ be the subalgebra of $\McN(X)$ 
generated by  $f\restrict_{X}$ and 
$g\restrict_{X}$. By \cite[4.1]{BuMu201X},
 $A$ is not strongly semisimple.
Let the map $\eta\colon X\to [0,1]^2$ be defined by $\eta=(f\restrict_{X},g\restrict_{X})$.
By   \cite[3.6]{Mu2011}, $A\cong \McN(\eta(X)),$ whence
$\McN(\eta(X))$ is not strongly semisimple.
By { Theorem~\ref{Thm:BusMun}}, 
$\eta(X)$  has a rationally outgoing \mbox{$1$-tangent}. 
 By 
 {  Lemma~\ref{lem:ReverseTang},} for some ${k\in\{1,\ldots,n-1\}}$
  $X$ has a rationally outgoing $k$-tangent. \qed
  
  \subsubsection*{Acknowledgment:} I would like to thank the valuable comments and suggestions made by Daniele Mundici on previous drafts of this paper.

\end{document}